\documentclass[11pt]{amsproc}

\usepackage[foot]{amsaddr}
\usepackage[utf8]{inputenc}
\usepackage[T1]{fontenc}
\usepackage{amsmath}
\usepackage{amsthm,enumerate} 
\usepackage{amssymb}
\usepackage{listings}
\usepackage{url}
\usepackage{hyperref}
\usepackage{footmisc}
\usepackage{booktabs}
\usepackage{soul}
\usepackage{xcolor}
\usepackage{dsfont}

\hypersetup{
    colorlinks,
    linkcolor={red!50!black},
    citecolor={blue!50!black},
    urlcolor={blue!80!black}
}

\usepackage{anysize}

\newtheorem{thm}{Theorem}[section]
\newtheorem{lemma}[thm]{Lemma}
\newtheorem{theorem}[thm]{Theorem}
\newtheorem{cor}[thm]{Corollary}
\newtheorem{prop}[thm]{Proposition}
\newtheorem{notation}[thm]{Notation}
\newtheorem{remark}[thm]{Remark}
\newtheorem{remarks}[thm]{Remarks}
\newtheorem{de}[thm]{Definition}
\newtheorem{example}[thm]{Example}

\newcommand{\PG}{{\rm PG}}
\newcommand{\Gauss}[2]{{\begin{bmatrix} #1 \\ #2 \end{bmatrix}_q}}
\newcommand{\allones}{\mathds{1}}
\newcommand{\charvec}{\mathds{1}}
\newcommand{\ff}{\mathbb{F}}
\newcommand{\EKR}{\mathcal{F}}

\renewcommand{\le}{\leqslant}
\renewcommand{\leq}{\leqslant}
\renewcommand{\ge}{\geqslant}
\renewcommand{\geq}{\geqslant}

\newtheorem{result}[thm]{Result}
\DeclareMathOperator{\CC}{{\mathcal C}}
\newcommand{\gauss}[2]{{#1\brack #2}}

\title[Maximum Erdős-Ko-Rado sets of chambers and their antidesigns]{Maximum Erd\H{o}s-Ko-Rado sets of chambers and their antidesigns in vector-spaces of even dimension}

\author[Heering]{Philipp Heering$^1$}
\author[Lansdown]{Jesse Lansdown$^2$}
\author[Metsch]{Klaus Metsch$^1$}

\address{$^1$Justus-Liebig-Universität, Mathematisches Institut, Arndtstraße 2, D-35392 Gießen, Germany
}
\address{$^2$School of Mathematics and Statistics, University of Canterbury, Christchurch, New Zealand}
\email{philipp.heering@math.uni-giessen.de}
\email{jesse.lansdown@canterbury.ac.nz}
\email{klaus.metsch@math.uni-giessen.de}

\begin{document}

\maketitle

\begin{abstract}
A chamber of the vector space $\ff_q^n$ is a set $\{S_1,\dots,S_{n-1}\}$ of subspaces of $\ff_q^n$ where $S_1\subset S_2\subset\hdots\subset S_{n-1}$ and $\dim(S_i)=i$ for $i=1,\dots,n-1$. By $\Gamma_n(q)$ we denote the graph whose vertices are the chambers of $\ff_q^n$ with two chambers $C_1=\{S_1,\dots,S_{n-1}\}$ and $C_2=\{T_1,\dots,T_{n-1}\}$ adjacent in $\Gamma_n(q)$, if $S_i\cap T_{n-i}=\{0\}$ for $i=1,\dots,n-1$.
The Erd\H{o}s-Ko-Rado problem on chambers is equivalent to determining the structure of independent sets of $\Gamma_n(q)$.
The independence number of this graph was determined in \cite{AlgebraicApproach} for $n$ even and given a subspace $P$ of dimension one, the set of all chambers whose subspaces of dimension $\frac n2$ contain $P$ attains the bound. The dual example of course also attains the bound. It remained open in \cite{AlgebraicApproach} whether or not these are all maximum independent  sets. Using a description from \cite{PreprintJanSamKlaus} of the eigenspace for the smallest eigenvalue of this graph, we prove an Erd\H{o}s-Ko-Rado theorem on chambers of $\ff_q^n$ for sufficiently large $q$, giving an affirmative answer for n even.
\end{abstract}

%\textbf{Keywords:}  Erd\H{o}s-Ko-Rado problem, q-analog of generalized Kneser graph, structure of maximal independent set

%\textbf{MSC(2020):} 
%05C69, %Dominating sets, independent sets, cliques
%05C35, % Extremal problems in graph theory

\section{Introduction}

Motivated by Erd\H{o}s, Ko, and Rado's fundamental result on maximum intersecting sets \cite{OriginalEKR}, an entire subfield of combinatorics has developed which seeks to answer the following questions in different settings: \textit{Given a suitable notion of ``intersection'', how large can a set of intersecting objects be? What is their structure?} Results of this type are referred to as \emph{Erd\H{o}s-Ko-Rado} theorems (or simply \emph{EKR} theorems) and have been considered for 
blocks in designs \cite{EKR_blocks_in_designs, An_extension_of_the_EKR_theorem_to_t-designs},
multisets  \cite{EKR_multisets}, 
integer sequences {\cite{EKR_integer_sequences}, 
 geometries \cite{Theorems_of_EKR_type_in_geometrical_settings},
and even permutations \cite{Intersection_theorems_in_permutation_groups,Intersecting_families_of_permutations}.
For an overview of the subject, see the recent survey by Ellis \cite{EllisSurvey3} or the book by Godsil and Meagher \cite{GodsilMeagher}.

One of the most natural and significant ways to generalize results on sets is to instead consider vector spaces over finite fields, where $k$-subsets become subspaces of dimension $k$. 
Indeed, the properties of sets often behave like the limit case of the vector space analogue. 
An EKR theorem for vector spaces was proved by Frankl and Wilson \cite{ekr_vectorspaces}.
More recently, interest has developed around EKR theorems for chambers of subspaces in finite vector spaces. Here the term \emph{chamber} is borrowed from the theory of buildings and refers to chains of contained subspaces.

More formally, let $V$ be a vector space of finite dimension $d$, then a \emph{flag} of $V$ is a set $F$ of subspacesof $V$ (other than $\{0\}$ and $V$) such that $U\le W$ or $W\le U$ for all $U,W\in F$ and $\{\dim(U)\mid U\in F\}$ is the \emph{type} of $F$. A \emph{chamber} $C$ of $V$ is a flag of type $\{1,\dots,d-1\}$ and we denote its $i$-subspace by $C_i$. Two chambers $C$ and $D$ of $V$ are \emph{opposite} if $C_i\cap D_{d-i}=\{0\}$ for $1\le i\le d-1$. An \emph{Erd\H{o}s-Ko-Rado set} (or \emph{EKR-set}) is a set of chambers for which no two are opposite. The following constructions in Example \ref{ex:classical} yield large EKR-sets, which we shall refer to as \emph{classical}.

\begin{example}\label{ex:classical}
Let $P$ be any $1$-dimensional subspace of $\ff_q^{2n}$, then an EKR-set is formed by the set of chambers whose $n$-dimensional subspace contains $P$. Dually, let $S$ be any $(2n-1)$-dimensional subspace of $\ff_q^{2n}$, then an EKR-set is formed by the set of chambers whose $n$-dimensional subspace is contained in $S$.
\end{example}

By taking the chambers as the vertices of a graph and adjacency as oppositeness, we obtain the \emph{Kneser graph on chambers}. Clearly an EKR-set is then a coclique in this graph, mirroring the original EKR problem in which an EKR-set is a coclique of the 
Kneser graph.
By computing the eigenvalues of the Kneser graph on chambers and applying the Delsarte-Hoffman bound, an upper bound on the size of an EKR-set was obtained in \cite{AlgebraicApproach}.
For odd-dimensional vector spaces no examples meeting the bound are known, however for even-dimension the classical examples meet the bound. Thus, we have the following important theorem on the size of a maximum EKR-set of chambers in a vector space of even dimension.

\begin{thm}[{\cite[Theorem 3.1]{AlgebraicApproach}}]\label{thm:maxsize}
Let $\EKR$ be an EKR-set of chambers of maximum size in $\ff_q^{2n}$, then
\[
|\EKR|=\frac{(q^{2n}-1)(q^{2n-1}-1)\ldots (q^{2}-1)(q-1)}{(q-1)^{2n}(q^n+1)}.
\]
\end{thm}

For $n=2$ an EKR theorem is given in \cite{heeringmetsch2023secondmax}
where maximum EKR-sets of chambers in $\ff_q^4$ are determined to be precisely the classical examples for $q\geq 4$. Note that Theorem \ref{thm:maxsize} resolves the question of maximum size of an EKR-set of chambers for finite vector spaces of even dimension, but it does not resolve the question of their structure. Indeed, determining their structure is a difficult question in general. In this paper we provide machinery for approaching this problem in the form of \emph{antidesigns} which are (weighted) sets of chambers that intersect a maximum EKR-set in a constant, determined only by the size of the EKR-set and the antidesign itself. We borrow the terminology of ``antidesigns'' from Delsarte theory in association schemes, since they are determined by the eigenspaces of the Kneser graph on chambers. 
However, for dimension at least $3$ the association scheme induced by the Kneser graph on chambers is neither symmetric nor commutative and so Delsarte theory is not directly applicable.
We provide key families of antidesigns in Section \ref{sec:antidesigns} and by combining them with geometric arguments involving weights we are able to prove that for sufficiently large $q$ (relative to $n$), all maximum EKR-sets of chambers are classical for even dimensional finite vector spaces.

\begin{thm}\label{thm:Bound}
For every integer $n \ge 3$, there exists an integer $m_n$ such that every maximum
Erdős-Ko-Rado set of chambers of $\ff_q^{2n}$ with $q\ge m_n$ is of classical type. 
\end{thm}

Explicit values for $m$ can be derived from the proof and are given for small $n$ at the end of Section \ref{Section: Classification of ekr}.
Using a spread based antidesign, we resolve the case $n=2$ for all $q$, thus completing the EKR theorem for $\ff_q^4$.

\begin{thm} \label{T: chambers of f_q^4}
 Every maximum Erdős-Ko-Rado set of chambers of $\ff_q^{4}$ is classical.
\end{thm}

\section{Preliminaries}

\begin{notation}\label{Def_chi_PG}
For integers $q,a,b\in\mathbb{Z}$ with $q\ge 2$ we define the Gaussian coefficient
 \begin{align*}
 \gauss{b}{a}_q=
 \begin{cases}
\displaystyle \prod_{i=1}^a\frac{q^{b-a+i}-1}{q^i-1} & \text{if $0\le a\le b$,}
\\
0 & \text{otherwise.} \end{cases}
\end{align*}
We omit the index $q$ if it is clear from the context.

For a prime power $q$ we denote the finite field with $q$ elements by $\ff_q$. For integers $c\ge 0$, we denote by $z_c(q)$ the number of chambers in $\ff_q^c$, but we write $z_c$ if $q$ is clear from the context. For integers $c\ge 0$ we define $[c]_q:=\prod_{i=1}^c(q^i-1)$ and omit the index $q$ if it is clear from the context.
\end{notation}

\begin{remarks}\label{firstremarks}
\begin{enumerate}[1.]
\item We have $z_0(q)=z_1(q)=1$, $[0]_q=1$ and $[1]_q=q-1$.
\item For every prime power $q$ and all $0\le a\le b$ we have $\gauss{b}{a}_q=\frac{[b]_q}{[a]_q[b-a]_q}$.
\end{enumerate}
\end{remarks}

\begin{lemma}\label{basicnumbers}
    Let $c\ge s\ge 0$ be integers. Then
\begin{enumerate}[(a)]
\item The number of $s$-subspaces of $\ff_q^c$ is $\gauss{c}{s}_q$.
\item  For $c\ge 0$ the number of chambers of $\ff_q^c$ is $z_c(q)=(q-1)^{-c}[c]_q$.
\item For $n\in\mathbb{N}$ we have $z_{2n}(q)=\gauss{2n}{n}_q\cdot z_n(q)^2$.
\item Every chamber of $\ff_q^c$ is opposite to exactly $q^{c\choose 2}$ chambers of $\ff_q^c$.
\item If $C$ is a chamber of $\ff_q^c$ and $S$ a subspace of dimension $s$ with $S\cap C_{c-s}=\{0\}$, then the number of chambers of $\ff_q^c$ that contain $S$ and are opposite to $C$ is $q^{{s\choose 2}+{c-s\choose 2}}$.
\end{enumerate}
\end{lemma}
\begin{proof}
\begin{enumerate}[(a)]
\item See Section 3.1 in \cite{Hirschfeld}
\item  
We use induction on $c$. For $c=0$ and $c=1$ the only chamber is the empty set. If $c\ge 2$, there exist $\gauss{c}{1}$ subspaces $U$ of dimension one. The induction hypothesis applied to the factor space $\ff_q^c/U$ shows that each such subspace lies in  $(q-1)^{-c+1}[c-1]$chambers. Hence the total number of chambers is  $(q-1)^{-c+1}[c-1]\gauss{c}{1}=(q-1)^{-c}[c]$.
\item This follows from (b) but can also be seen geometrically, since $\gauss{2n}{n}$ is the number of $n$-subspaces of $\ff_q^{2n}$ and each $n$-subspace is a member of $z_n(q)^2$ chambers. 
\item We prove this by induction on $c$, the cases $c=0,1$ being obvious. Now let $c\ge 2$ and $C$ be a chamber of $\ff_q^c$. If $D$ is a chamber opposite to $C$, then its $1$-subspace $D_1$ is not contained in $C_{c-1}$. From (a) we see that there are $\gauss{c}{1}-\gauss{c-1}{1}=q^{c-1}$ subspaces $U$  of dimension one not contained in $C_{c-1}$. In the quotient space $\ff_q^c/U$, we see that each such subspace $U$ can be extended in exactly $q^{c-1\choose 2}$ chambers that are opposite to $C$. Hence $C$ is opposite to $q^{c-1}q^{c-1\choose 2}=q^{c\choose 2}$ chambers.
\item
The set $\underline{C}:=\{S\cap C_i\mid c-s<i<c\}$ is a chamber of $S$. Part (c) of Lemma \ref{basicnumbers} shows that $S$ has exactly $q^{s\choose 2}$ chambers $D$ opposite to $\underline C$. The set $\bar C:=\{(S+C_i)/S\mid 0<i<c-s\}$ is a chamber in the factor space $\ff_q^c/S$ and Lemma \ref{basicnumbers} shows that $\ff_q^c/S$ has exactly $q^{c-s\choose 2}$ chambers opposite to $\bar C$. Such a chamber has the form $\{U_i/S\mid s<i<c\}$ where $U_i$ is a subspace of dimension $s+i$ containing $S$. For such a chamber, the union of $\{U_i\mid s<i<c\}$ with any of the above chambers $D$ is a chamber opposite to $C$ and all chambers opposite to $C$ have this form.
\end{enumerate}
\end{proof}

\begin{lemma}\label{numberofsubspaces}
For integers $a,b,d\ge 0$ with $d\ge a+b$, consider a subspace $A$ of dimension $a$ in an $\ff_q$ vector space $T$ of dimension $d$. Then the number of $b$-subspaces of $T$ that intersect $A$ trivially is $\gauss{d-a}{b}q^{ab}$.
\end{lemma}
\begin{proof}
See Section 3.1 in \cite{Hirschfeld}.
\end{proof}

\begin{lemma}\label{ChambeerExtensions}
Consider in $\ff_q^d$ a non-empty flag $F$ and let $\{t_1,t_2,\dots,t_f\}$ with $f=|F|$ be its type where $t_1<t_2<\dots<t_f$. Then $F$ is contained in exactly $z_{t_1}z_{d-t_f}\prod_{i=2}^f(z_{t_i-t_{i-1}})$ chambers of $\ff_q^d$.
\end{lemma}
\begin{proof}
Let $T_i$ be the subspace of $F$ that has dimension $t_i$. Put $T_0:=\{0\}$, $T_{f+1}:=\ff_q^d$, $t_0=0$ and $t_{f+1}=d$. Then $t_i$ is the dimension of $T_i$ for $i=0,\hdots,f+1$. By Lemma \ref{basicnumbers} the number of chambers of $T_i/T_{i-1}$ is $z_{t_i-t_{i-1}}$, $i=1,\dots,f+1$. It follows that the number of chambers containing $F$ is equal to $\prod_{i=1}^{f+1}(z_{t_i-t_{i-1}})$.
\end{proof}

\section{Antidesigns}\label{sec:antidesigns}

For an integer $d\ge 1$ and the finite field $\ff_q$ we define the Kneser graph on chambers of $\ff_q^{d}$ to be the graph whose vertices are the chambers of $\ff_q^d$ where two vertices are adjacent if and only if the corresponding chambers are opposite. In this paper we denote this graph by $\Gamma_d(q)$. For positive integers $n\ge 1$ it was shown in \cite{AlgebraicApproach} that the smallest eigenvalue of $\Gamma_{2n}(q)$ is $\lambda:=-q^{2n(n-1)}$ and that the corresponding eigenspace has dimension $n\cdot\frac{q^{2n}-q}{q-1}$. Applying the Hoffmann bound, the authors of \cite{AlgebraicApproach} showed that the independence number of $\Gamma_{2n}(q)$ is (see also Theorem \ref{thm:maxsize})
\begin{align*} %\label{eqn_coclique_number}
\alpha(\Gamma_{2n}(q))=\frac{z_{2n}(q)}{q^n+1}.
\end{align*}

If $d$ is the number of chambers of $\Gamma_{2n}(q)$, then we index the entries of a vector of $\mathbb{Q}^d$ by the chambers of $\Gamma_{2n}(q)$. We also view a vector $v$ of $\mathbb{Q}^d$ as a map from the set of all chambers to $\mathbb{Q}$; then $v(C)$ is the entry of $v$ in the position corresponding to $C$ for all chambers $C$. For a subset $X$ of chambers we define its \emph{characteristic vector} $\charvec_X$ by $\charvec_X(C)=1$, if $C\in X$, and $\charvec_X(C)=0$, if $C$ is a chamber that is not in $X$. 
We shall denote the all ones vector of $\mathbb{Q}^d$ by $\allones$, which is the characteristic vector of the set of all chambers. 

For every independent set $\EKR$ of $\Gamma_{2n}(q)$ satisfying $|\EKR|=\alpha(\Gamma_{2n}(q))$ we have $\charvec_\EKR\in\langle \allones \rangle+E$ where $E$ is the eigenspace of $\Gamma_{2n}(q)$ for its smallest eigenvalue 
(see Theorem 2.4.1 in \cite{GodsilMeagher}).
An explicit description of $E$ is thus a helpful tool for classifying all largest independent sets. One possibility is to use such a description to prove that certain vectors are antidesigns, where we use the following notation in this paper.

\begin{de}
Let $E$ be the eigenspace for the smallest eigenvalue of $\Gamma_{2n}(q)$. An \emph{antidesign} of $\Gamma_{2n}(q)$ is a function $v:V(\Gamma_{2n}(q))\to\mathbb{Q}$ such that $e^\top v=0$ for all $e\in E$.
\end{de}

The reason why we are interested in antidesigns is the following well-known fact, which will enable us to derive geometric insight about largest independent sets. Though its proof is well-known, we give the short argument.

\begin{prop}\label{Intersection_F_with_antidesing}
If $\EKR$ is a largest independent set of $\Gamma_{2n}(q)$ and if $v$ is an antidesign, then
\begin{align*}
\charvec_\EKR^\top v=\frac{\allones^\top v}{q^n+1}.
\end{align*}
\end{prop}
\begin{proof}
If $E$ is the eigenspace of the smallest eigenvalue of $\Gamma_{2n}(q)$, then $\charvec_\EKR\in\langle \allones \rangle+E$, that is $\charvec_\EKR=k\allones+e$ for some $k\in\mathbb{Q}$ and $e\in E$. We write $v=\ell \allones+w$ with $\allones^\top w=0$. Since $e^\top v=0$, $\allones^\top \charvec_\EKR=k\allones^\top \allones$ and $\allones^\top v=\ell \allones^\top \allones$, then
\begin{align*}
\charvec_\EKR^\top v =(k\allones+e)^\top(\ell \allones+w)=k\ell \allones^\top \allones =\frac{(\allones^\top\charvec_\EKR)(\allones^\top v)}{\allones^\top \allones}.
\end{align*}
The assertion follows from the fact that $\allones^\top \allones$ is the number of vertices, which is $z_{2n}$ and $\allones^\top\charvec_\EKR$ is the size of $\EKR$ which is $\frac{z_{2n}}{q^n+1}$ by 
Theorem \ref{thm:maxsize} and Lemma \ref{basicnumbers}.
\end{proof}

In order to prove that certain functions are antidesigns, we need the following explicit description of the eigenspace $E$, which was proved in \cite{PreprintJanSamKlaus}.

\begin{notation}
Let $n\ge 1$ be an integer and let $\CC$ be the set of chambers of $\ff_q^{2n}$. For each subspace $P$ of dimension one and each $i\in\{1,\dots,n\}$, define the map $\chi^i_P:\CC\to\mathbb{Q}$ where
\begin{align*}
\chi^i_P(C)=\left\{\begin{array}{rl}
q^i & \text{if $P\le C_n$ and $P\not\le C_{n-i}$}
\\
-1 &  \text{if $P\le C_{n+i}$ and $P\not\le C_n$}
\\0 & \text{if $P\le C_{n-i}$ or $P\not\le C_{n+i}$}
\end{array}
\right.
\end{align*}
for all $C\in\CC$.
\end{notation}

\begin{thm}\label{PreprintJanSamKlaus}
The eigenspace for the smallest eigenvalue of $\Gamma_{2n}(q)$, $n\ge 1$, is spanned by the vectors $\chi^i_P$ where $i\in\{1,\dots,n\}$ and $P$ is a subspace of dimension one of $\ff_q^{2n}$.
\end{thm}

This section is devoted to the construction of antidesigns. In our first family the antidesigns are based on a spread.
Recall that the vertices of $\Gamma_{2n}(q)$ are the chambers of $\ff_q^{2n}$.

\subsection{Spread based antidesigns}

An \emph{$n$-spread} of $\ff_q^{2n}$ is  a set $S$ of subspaces of rank $n$ such that every $1$-subspace is a subspace of a unique member of $S$. It follows that $|S|=q^n+1$. Spreads exist for all $n$, e.g. consider the set $S_0$ of all $q^n+1$ subspaces of dimension one of $\ff_{q^n}^2$. If we understand $\ff_{q^n}^2$ as an $\ff_q$-subspace, then it has dimension $2n$ and $S_0$ is a spread of $n$-subspaces of it. A \emph{$t$-fold $n$-spread} is a set $S$ of subspaces of rank $n$ such that every $1$-subspace is a subspace of exactly $t$ subspaces of $S$. It follows that $|S|=t(q^n+1)$. Of course, the disjoint union of $t$ spreads is a $t$-fold $n$-spread but there exist others.

\begin{notation}\label{SpreadAntidesign}
Let $n\ge 1$ be an integer and let $S$ be a $t$-fold $n$-spread of $\ff_q^{2n}$. Let $\CC$ be the set of chambers of $\ff_q^{2n}$ and let $v_S$ be the map from $\CC$ to $\mathbb{Q}$ with
\begin{align*}
v_S(C)=\left\{\begin{array}{ll}
1 & \text{if $C_n\in S$,}
\\
0 &  \text{if $C_n\notin S$,}
\end{array}
\right.
\end{align*}
for all chambers $C$.
\end{notation}

\begin{theorem}\label{thm:SpreadAntidesigns}
For every t-fold n-spread $S$ of $\ff_q^{2n}$, the map $v_S$ is an antidesign with $\allones^\top v_S=t(q^n+1)z_n(q)^2$.
\end{theorem}
\begin{proof}
In view of Theorem \ref{PreprintJanSamKlaus} we have to show $v_S^\top \chi^i_P=0$. If $\mathcal{C}$ is the set of all chambers we have $v_S^\top \chi^i_P=\sum_{C\in\mathcal{C}}v_f(C)\chi^i_P(C)$.
Consider $i\in\{1,\dots,n\}$ and a $1$-subspace $P$. 
Then $P$ is a subspace of exactly $t$ members $U_1,\dots,U_t$ of $S$. By Lemma \ref{numberofsubspaces} each of these contains $\gauss{n-1}{n-i}q^i$ subspaces $V$ of dimension $n-i$ with $P\not\le V$. Given such spaces $U$ and $V$ with $U\le V$, Lemma \ref{ChambeerExtensions} shows that there exist $z_{n-i}z_iz_n$ chambers that contain $V$ und $U_i$. Hence, the number of chambers $C$ satisfying $v_S(C)=1$ and $\chi^i_P(C)=q^i$, that is $C_n\in\{U_1,\dots,U_t\}$ and $P\not\le C_{n-i}$ is
\begin{align*}
A:=t\cdot \gauss{n-1}{n-i}q^{n-i}z_iz_{n-i}z_n.
\end{align*}
A subspace $U\in S\setminus\{U_1,\dots,U_t\}$ does not contain $P$ and hence it is contained in $\gauss{n-1}{i-1}$ subspaces $W$ of dimension $n+i$ with $P\le W$. Lemma \ref{ChambeerExtensions} implies that for each pair $(U,W)$ there exist $z_{n-i}z_iz_n$ chambers that contain $U$ und $W$. Hence the number of chambers satisfying $v_S(C)=1$ and $\chi^i_P(C)=-1$, that is $C_n\in S\setminus\{U_1,\dots,U_t\}$ and $P\le C_{n+i}$ is
\begin{align*}
B:=(|S|-t)\gauss{n-1}{i-1}z_iz_{n-i}z_n=t\gauss{n-1}{i-1}q^nz_iz_{n-i}z_n.
\end{align*}
Since $v_S^\top\chi^i_P=Aq^i-B$, it follows that $v_S^\top\chi^i_P=0$.
Since this holds for all $i$ and $P$, Theorem \ref{PreprintJanSamKlaus} implies that $v_S$ is an antidesign. Since every $n$-subspace lies in $z_n(q)^2$ chambers of $\ff_q^{2n}$, it follows that we have $\allones^\top v_S=|S|\cdot z_n(q)^2=t(q^n+1)z_n(q)^2$.
\end{proof}

One example for an $t$-fold $n$-spread is given by the maximal totally isotropic subspaces with respect to a non-degenerate symplectic form. In the next subsection we introduce a second possibility to construct an antidesign from a symplectic form.
The following corollary of Theorem \ref{thm:SpreadAntidesigns} follows immediately from Proposition \ref{Intersection_F_with_antidesing}.

\begin{cor} \label{C: spreads intersection size}
If $S$ is a $t$-fold $n$-spread of  $\PG(2n,q)$ and $\EKR$ is a maximum EKR-set of $\Gamma_{2n}(q)$, then $v_S^\top \charvec_\EKR=t\cdot z_n(q)^2$.
\end{cor}

\subsection{Antidesigns from alternating forms}

Let $f$ be a non-degenerate alternating form on the $\ff_q$ vector space $\ff_q^{2n}$. The maximal totally isotropic subspaces have dimension $n$, see Section 6.5 in \cite{Cameron}. Let $\perp$ denote the related polarity.
We call a chamber $C$ of $\ff_q^{2n}$ symplectic with respect to $f$, if $C_n$ is totally isotropic (and hence $C_i$ is for all $i\le n$) and if $C_{2n-i}=C_i^\perp$ for $i=1,\dots,n$.

\begin{notation}
Let $n\ge 1$ be an integer and let $v_f$ be the map from the set of all chambers of $\ff_q^{2n}$ to $\mathbb{Q}$ with $v_f(C)=1$, if the chamber $C$ is symplectic, and $v_f(C)=0$ otherwise.
\end{notation}

\begin{theorem}
For a non-degenerate alternating form $f$ on the $\ff_q$ vector space $\ff_q^{2n}$, the map $v_f$ is an antidesign with 
$\allones^\top v_f=z_n(q)\prod_{i=1}^{n}(q^i+1)$.
\end{theorem}
\begin{proof}
From Theorem 6.14 in \cite{Cameron} we see that the number of totally isotropic $n$-subspaces is $g_n:=\prod_{i=1}^n(q^i+1)$ and that every $1$-subspace is contained in $g_{n-1}:=\prod_{i=1}^{n-1}(q^i+1)$ of these.

In view of Theorem \ref{PreprintJanSamKlaus} we have to show that $v_f^\top \chi^i_P=0$ for all $i=1,\dots,n$ and for every $1$-subspace $P$ of $\ff_q^{2n}$. Given $P$ and $i$, we first count the number $A$ of chambers $C$ of $\ff_q^{2n}$ that satisfy $\chi^i_P(C)=q^i$ and $v_f(C)=1$, that is the number of symplectic chambers satisfying $P\le C_n$, $P\not\le C_{n-i}$.
There are $g_{n-1}$ totally isotropic $n$-subspaces $U$ containing $P$ and each such $U$ contains $\gauss{n-1}{n-i}q^{n-i}$ subspaces $V$ of dimension $n-i$ that do not contain $P$. Lemma \ref{ChambeerExtensions} implies that for every $U$ and $V$ with $V\le U$ there are $z_iz_{n-i}$ symplectic chambers containing $U$ and $V$. This gives 
\begin{align*}
A=g_{n-1}\gauss{n-1}{n-i}q^{n-i}z_iz_{n-i}.
\end{align*}

Next we count the number $B$ of chambers $C$ of $\ff_q^{2n}$ that satisfy $\chi^i_P(C)=-1$ and $v_f(C)=1$, that is the number of symplectic chambers satisfying $P\le C_{n+i}$ and $P\not\le C_n$, which is equivalent to $C_{n-i}\le P^\perp$ and $P\not\le C_n$. There are $g_n-g_{n-1}=q^ng_{n-1}$  totally isotropic $n$-subspaces $U$ that do not contain $P$. For such a subspace $U$ the subspace $P^\perp\cap U$ has dimension $n-1$ and hence $U$ contains $\gauss{n-1}{n-i}$ subspaces $V$ of dimension $n-i$ with $V\le P^\perp$. Lemma \ref{ChambeerExtensions} shows that each of the resulting flags $\{V,U\}$ is contained in $z_iz_{n-i}$ symplectic chambers. This gives 
\begin{align*}
B=q^ng_{n-1}\gauss{n-1}{n-i}z_iz_{n-i}.
\end{align*}

Since $v_f^\top\chi^i_P=Aq^i-B$, it follows that $v_f^\top\chi^i_P=0$. 
This completes the proof that $v_f$ is an antidesign. 
Finally, we have that $\allones^\top v_f$ is the number of symplectic chambers in $\ff_q^{2n}$. It follows that $\allones^\top v_f=z_n(q)\cdot g_n$. 
\end{proof}

The following corollary follows immediately from Theorem \ref{thm:SpreadAntidesigns} and Proposition \ref{Intersection_F_with_antidesing}.

\begin{cor}\label{cor:symplectic}
If $f$ is a non-degenerate alternating form on the $\ff_q$ vector space $\ff_q^{2n}$ and $\EKR$ is a maximum EKR-set of $\Gamma_{2n}(q)$, then $v_f^\top\charvec_\EKR=z_n(q)\cdot \prod_{i=1}^{n-1}(q^i+1)$.
\end{cor}

\begin{remark}
Note that the set of symplectic chambers is in one-to-one correspondence with the chambers of the polar space and moreover the definition of oppositeness agrees \cite{AlgebraicApproach}. Hence an EKR-set of chambers in $\ff_q^{2n}$ induces an EKR-set of chambers in the polar space, and if the EKR-set in $\ff_q^{2n}$ is a maximum one, then the size of the induced EKR is given by Corollary \ref{cor:symplectic} which is maximal in the polar space \cite{AlgebraicApproach}.
\end{remark}

\subsection{Antidesigns from unitary forms}

In this subsection we assume that the order $q$ of the field $\ff_q$ is a square, and we let $f$ be a non-degenerate unitary form on the $\ff_q$ vector space $\ff_q^{2n}$. Let $\perp$ denote the related polarity. Subspaces $U$ of $\ff_q^{2n}$ with $f(u,v)=0$ for all $u,v\in U$ are called \emph{totally isotropic}; this is equivalent to $U\le U^\perp$. The maximal totally isotropic subspaces will be called \emph{generators}, by Lemma 23.3.1 of \cite{Hirschfeld&Thas} they have dimension $n$.
By $t_n(q)$ we denote the number of generators of $\ff_q^{2n}$ with respect to a non-degenerate unitary form. The following lemma is well-known.

\begin{lemma}
We have 
\begin{align*}
t_n(q)=\prod_{j=1}^n(\sqrt q\cdot q^{j-1}+1).
\end{align*}
If $f$ is a non-degenerate unitary form on $\ff_q^{2n}$, then each totally isotropic subspace of dimension one is contained in $t_{n-1}(q)$ generators. \end{lemma}
\begin{proof}
The formula for $t_n(q)$ can be found in Lemma 23.3.2 of \cite{Hirschfeld&Thas}. Let $P$ be a subspace of dimension one. If $P$ is totally isotropic, then $f$ induces a non-degenerate unitary form in the factor space $P^\perp/P\simeq \ff_q^{2n-2}$, which implies that $P$ lies on $t_{n-1}(q)$ generators. 
\end{proof}

We call a chamber of $\ff_q^{2n}$ \emph{hermitian} with respect to $f$, if $C_{2n-i}=C_i^\perp$ for $i=1,\dots,n$, which implies that $C_1,\dots,C_n$ are totally isotropic. These are in one-to-one correspondence with the chambers of the polar space and moreover the definition of oppositeness agrees \cite{AlgebraicApproach}.

Recall that $z_n(q)$ is the number of chambers of $\ff_q^n$. 

\begin{notation}\label{HermitianAntidesign}
Let $n\ge 1$ be an integer and let $f$ be a non-degenerate unitary form on $\ff_q^{2n}$. Put $k:=z_n(q)-1$. Let $\CC$ be the set of chambers of $\ff_q^{2n}$ and let $v_f$ be the map from $\CC$ to $\mathbb{Q}$ with
\begin{align*}
v_f(C)=\left\{\begin{array}{rl}
-k & \text{if $C_{n+i}=C_{n-i}^\perp$ for $i=0,\dots,n-1$,}
\\
1 & \text{if $C_n=C_n^\perp$ but $C_{n+i}\not=C_{n-i}^\perp$ for some $i\in\{1,\dots,n-1\}$,}
\\
0 &  \text{if $C_n\not=C_n^\perp$,}
\end{array}
\right.
\end{align*}
for all chambers $C$.
\end{notation}

\begin{theorem}
For a non-degenerate unitary form $f$ on the $\ff_q$ vector space $
\ff_q^{2n}$, $q$ a square, the map $v_f$ is an antidesign satisfying $\allones^\top  v_f=0$.
\end{theorem}
\begin{proof}
We write $t_i$ for $t_i(q)$.
The number of chambers $C$ with $C_n=C_n^\perp$, that is for which $C_n$ is totally isotropic, is $t_n\cdot z_n(q)^2$. Exactly $t_nz_n(q)$ of these satisfy $C_{n+i}=C_i^\perp$ for $i=0,\dots,n-1$, and the remaining $t_nz_n(q)(z_n(q)-1)$ do not. Since $k=z_n(q)-1$, it follows that $\allones^\top  v_f=0$.
%} 

We now prove that $v_f$ is an antidesign.
In view of Theorem \ref{PreprintJanSamKlaus} this requires
us
to show that $v_f^\top \chi^i_P=0$ for all $i=1,\dots,n$ and all $1$-subspace $P$ of $\ff_q^{2n}$. Recall that $v_f^\top \chi^i_P=\sum_{\mathcal{C}}v_f(C)\chi^i_P(C)$ where the sum is over all chambers. For the calculation of this sum, we distinguish two cases. Recall that $\chi^i_P\in\{0,-1,q^i\}$.
\bigskip

\textbf{Case 1.} $P$ is not totally isotropic. 

Consider a chamber $C$ with $\chi^i_P(C)\not=0$ and $v_f(C)\not=0$. Then $C_n$ is totally isotropic, so $P\not\le C_n$, hence $\chi^i_P(C)=-1$, that is $P\le C_{n+i}$ and $P\not\le C_n$. 

The number of these chambers $C$ is $A:=t_n\gauss{n-1}{i-1}z_nz_iz_{n-i}$ for the following reasons. There are $t_n$ generators $U$ and each generator is contained in $\gauss{n-1}{i-1}$ subspaces $V$ of dimension $n+i$ that also contain $P$. Furthermore, Lemma \ref{ChambeerExtensions} implies that each flag $\{U,V\}$ extends in $z_nz_iz_{n-i}$ ways to a chamber. Let $B$ be the number of these chambers that are hermitian.

We calculate $B$. If $C$ is a hermitian chamber with $P\le C_{n+i}$, then $P\le C_{n-i}^\perp$, that is $C_{n-i}\le P^\perp$. Given a generator $U$, then $P\not\le U=U^\perp$, so $P^\perp\cap U$ has dimension $n-1$ and contains $\gauss{n-1}{n-i}$ subspaces $V$ of dimension $n-i$. For these $P\le V^\perp$.
Lemma \ref{ChambeerExtensions} states that for every $U$ and $V$ with $V\le U$ and a generator $U$, there exists $z_iz_{n-i}$ hermitian chambers containing $U$ and $V$. Hence $B:=t_n\gauss{n-1}{n-i}z_iz_{n-i}$. It follows that 
\begin{align*}
v_f^\top \chi^i_P&=
B\cdot k+(A-B)\cdot (-1)=-A+B(k+1)
\\
&=t_n\gauss{n-1}{i-1}z_iz_{n-i}(k+1-z_n)
\end{align*}
and this is zero, since $k=z_n-1$.
\bigskip
      
\textbf{Case 2.} $P$ is totally isotropic.

\begin{enumerate}[1.]
\item We count the number $A_{11}$ of chambers $C$ of $\ff_q^{2n}$ that satisfy $\chi^i_P(C)=q^i$ and $v_f(C)=-k$, that is the number of hermitian chambers satisfying $P\le C_n$, $P\not\le C_{n-i}$.

There are $t_{n-1}$ generators $U$ containing $P$ and each such $U$ contains $\gauss{n-1}{n-i}q^{n-i}$ subspaces $V$ of dimension $n-i$ that do not contain $P$. For every $U$ and $V$ with $V\le U$ there are $z_iz_{n-i}$ hermitian chambers containing $U$ and $V$. This gives 
\begin{align*}
A_{11}=t_{n-1}\gauss{n-1}{n-i}q^{n-i}z_iz_{n-i}.
\end{align*}

\item We count the number $A_{12}$ of chambers $C$ of $\ff_q^{2n}$ that satisfy $\chi^i_P(C)=q^i$ and $v_f(C)=1$, that is the number of chambers that are not hermitian and satisfy $P\le C_n$, $P\not\le C_{n-i}$.
    
    The number of all chambers $C$ with $P\le C_n$, $P\not\le C_{n-i}$ and for which $C_n$ is a generator is $t_{n-1}\gauss{n-1}{n-i}q^{n-i}z_iz_{n-i}z_n$, since there are $t_{n-1}$ choices for the generator and each contains $\gauss{n-1}{n-i}q^{n-i}$ subspaces of rank $n-i$ that do not contain $P$. Since $A_{11}$ of these are hermitian, it follows that
    \begin{align*}
    A_{12}&=t_{n-1}\gauss{n-1}{n-i}q^{n-i}z_iz_{n-i}z_n- A_{11} %t_{n-1}\gauss{n-1}{n-i}q^{n-i}z_iz_{n-i}
    \\
    &=t_{n-1}\gauss{n-1}{n-i}q^{n-i}z_iz_{n-i}(z_n-1).
    \end{align*}

\item We count the number $A_{21}$ of chambers $C$ of $\ff_q^{2n}$ that satisfy $\chi^i_P(C)=-1$ and $v_f(C)=-k$, that is the number of hermitian chambers satisfying $P\le C_{n+i}$ and $P\not\le C_n$, which is equivalent to $C_{n-i}\le P^\perp$ and $P\not\le C_n$. 
    
    There are $t_n-t_{n-1}=\sqrt qq^{n-1}t_{n-1}$  generators $U$ that do not contain $P$. For such a generator $U$ the subspace $P^\perp\cap U$ has dimension $n-1$ and hence $U$ contains $\gauss{n-1}{n-i}$ subspaces $V$ of dimension $n-i$ with $V\le P^\perp$. Each of the resulting flags $\{V,U\}$ is contained in $z_iz_{n-i}$ hermitian chambers. This gives 
    \begin{align*}
 A_{21}=\sqrt qq^{n-1}t_{n-1}\gauss{n-1}{n-i}z_iz_{n-i}. 
\end{align*}

\item We count the number $A_{22}$ of chambers $C$ of $\ff_q^{2n}$ that satisfy $\chi^i_P(C)=-1$ and $v_f(C)=1$, that is the number of chambers $C$ that are not hermitian such that $C_n$ is a generator, $P\le C_{n+i}$ and $P\not\le C_n$.
    
    The number of all chambers $C$ such that $C_n$ is a generator, $P\le C_{n+i}$ and $P\not\le C_n$ is equal to $\sqrt qq^{n-1}t_{n-1}\gauss{n-1}{i-1}z_iz_{n-1}z_n$, since there are $t_n-t_{n-1}=\sqrt qq^{n-1}t_{n-1}$ generators $U$ that do not contain $P$ and each is contained $\gauss{n-1}{i-1}$ subspaces of dimension $n+i$ that contain $P$. Since $A_{21}$ of these are hermitian, it follows that
    \begin{align*}
    A_{22}&=\sqrt qq^{n-1}t_{n-1}\gauss{n-1}{i-1}z_iz_{n-1}z_n-A_{21}
    \\
    &=\sqrt qq^{n-1}t_{n-1}\gauss{n-1}{i-1}z_iz_{n-i}(z_n-1)
    \end{align*}
\end{enumerate}
Finally, we have 
\begin{align*}
v_f^\top \chi^i_P&=
A_{11}\cdot (-kq^i)+A_{12}\cdot q^i+A_{21}\cdot k+A_{22}\cdot (-1)
\end{align*}
and this is zero, since $A_{11}=kA_{12}$ and $A_{21}=kA_{22}$.

Hence, $v_f^\top \chi^i_P=0$ in Case 1 as well as in Case 2. Therefore $v_f$ is an antidesign.
\end{proof}

\begin{cor}\label{cor:unitary}

If $f$ is a non-degenerate unitary form on the $\ff_q$ vector space $\ff_q^{2n}$, $q$ square, and $\EKR$ is a maximum EKR-set of $\Gamma_{2n}(q)$, then $v_f^\top\charvec_\EKR=0$.
\end{cor}

\subsection{Subspace based antidesigns}

\begin{notation}
We say that two subspaces meet if they intersect non-trivially. On the other hand, we say that two subspaces are skew if they intersect trivially.
\end{notation}

In this subsection we introduce the antidesigns that will play a prominent role in the characterizations of the largest cocliques of $\Gamma_{2n}(q)$.

\begin{notation}\label{SubspaceAntidesign}
Let $n\ge 1$ be an integer and let $\CC$ be the set of chambers of $\ff_q^{2n}$. For $s\in\{1,\dots,n\}$ and each subspace $S$ of dimension $s$, we define $v_S:\CC\to \mathbb{Q}$ such that
\begin{align*}
v_S(C)=\left\{\begin{array}{ll}
q^{s(2n-s)-n} & \text{if $S=C_s$,}
\\
1 &  \text{if $S\cap C_{2n-s}=\{0\}$,}
\\0 & \text{otherwise.}
\end{array}
\right.
\end{align*}
\end{notation}

In order to show that this is an antidesign, we will show that $v_Q^\top \chi^i_P=0$ for all $i$ and $P$. We distinguish two cases, which will be handled separately in the next two lemmata.

\begin{lemma}\label{L: subspace antidesign 1}
For each $1$-subspace $P$ and each subspace $S$ with $1\le s:=\dim(S)\le n$ and $P\le S$, we have $v_S\chi^i_P=0$ for all $i\in\{1,\dots,n\}$.
\end{lemma}
\begin{proof}
\begin{enumerate}[1.]
\item We count the number of chambers $C$ with $C_s=S$, $P\le C_n$ and $P\not\le C_{n-i}$.

Since $P\le S$, this number is zero when $s\le n-i$. Now consider the situation when $s\ge n-i+1$.

Then $S$ has $\gauss{s-1}{n-i}\cdot q^{n-i}$ subspaces of dimension $n-i$ that do not contain $P$. For each such subspace $X$, the flag $\{X,S\}$ can be extended in $z_{n-i}z_{s+i-n}z_{2n-s}$ ways to a chamber fulfilling the requirements. Hence the number of chambers in question is
\begin{align*}
A:=\gauss{s-1}{n-i}\cdot q^{n-i}\cdot z_{n-i}\cdot z_{s+i-n}\cdot z_{2n-s}
\end{align*}

\item We count the number of chambers $C$ with $C_s=S$, $P\le C_{n+i}$ and $P\not\le C_n$.

Since $P\le S$ and each chamber $C$ satisfies $C_s\le C_n$, this number is zero.

\item We count the number of chambers $C$ such that $C_{2n-s}\cap S=\{0\}$, $P\le C_n$ and $P\not\le C_{n-i}$.

Since $P\le S$ and  each chamber $C$ satisfies $C_n\le C_{2n-s}$, this number is zero.

\item We count the number of chambers $C$ such that $C_{2n-s}\cap S=\{0\}$ and $P\le C_{n+i}$ and $P\not\le C_n$.

If $n+i\le 2n-s$, then each chamber $C$ satisfies $C_{n+i}\le C_s$ and hence no chamber has the desired properties, so its number is zero.

Now consider the situation when $2n-s\le n+i-1$, that is $s\ge n-i+1$. Then there exist $q^{s(2n-s)}$ subspaces $U$ of dimension $2n-s$ with $U\cap S=\{0\}$. For each such subspace $U$ there are $\gauss{2n-(2n-s+1)}{n+i-(2n-s+1)}=\gauss{s-1}{s-n+i-1}$ subspaces $V$ of dimension $n+i$ containing $U$ and $P$. Finally, each such flag $\{U,V\}$ can be extended to $z_{2n-s}z_{s+i-n}z_{n-i}$ distinct chambers, which all fulfill the requirements. Hence the number in question is
\begin{align*}
B:=q^{s(2n-s)}\cdot \gauss{s-1}{s-n+i-1}\cdot z_{2n-s}\cdot z_{s+i-n}\cdot z_{n-i}.
\end{align*}
\end{enumerate}
We have $v_S\chi^i_P=q^{s(2n-s)-n}A-B$. Since the Gaussian coefficients occurring in $A$ and $B$ are equal, it follows that $v_S\chi^i_P=0$.
\end{proof}

\begin{lemma} \label{L: subspace antidesign 2}
For each point $P$ and each subspace $S$ with $1\le s:=\dim(S)\le n$ and $P\not\le S$, we have $v_S\chi^i_P=0$ for all $i\in\{1,\dots,n\}$. For each point $P$ and each subspace $S$ of dimension at most $n$ with $P\not\le S$, we have $v_S\chi^i_P=0$ for all $i\in\{1,\dots,n\}$.
\end{lemma}
\begin{proof}
\begin{enumerate}[1.]
\item We count the number of chambers $C$ with $C_s=S$, $P\le C_n$ and $P\not\le C_{n-i}$.

There are $\gauss{2n-s-1}{n-s-1}$ subspaces $U_n$ of dimension $n$ satisfying $P,S\le U_n$. Notice that this is zero if $s=n$.

If $n-i\le s$, then the flag $\{S,U_n\}$ extends to exactly $z_sz_{n-s}z_n$ chambers and all these have the desired properties.

If $s\le n-i$, then for each subspace $U_n$ as above, there we see in the quotient space on $S$, that there exist $\gauss{n-s-1}{n-i-s}q^{n-i-s}$ subspaces $U_{n-i}$ satisfying $S\le U_{n-i}\le U_n$ and $P\notin U_{n-i}$, and each flag $\{S,U_{n-i},U_n\}$ extends to $z_sz_{n-i-s}z_iz_n$ chambers.

Using the symmetry of the Gaussian coefficients we find
\begin{align*}
A_{11}=\begin{cases}
\displaystyle \gauss{2n-s-1}{n-s-1}\gauss{n-s-1}{n-i-s}q^{n-i-s}z_sz_{n-i-s}z_iz_n & \mbox{if $s+i\le n$,}
\\\\
\displaystyle \gauss{2n-s-1}{n-s-1}z_sz_{n-s}z_n & \mbox{if $s+i\ge n$}
\end{cases}
\end{align*}
and of course, both values coincide when $s+i=n$.

\item We count the number of chambers $C$ with $C_s=S$, $P\le C_{n+i}$ and $P\not\le C_n$.

There exist $\gauss{2n-s-1}{n-s}q^{n-s}$ subspaces $U_n$ with $S\le U_n$ and $P\not\le U_n$, and for each such there exist $\gauss{2n-n-1}{i-1}$ subspaces $U_{n+i}$ of dimension $n+i$ satisfying $U_n,P\le U_{n+i}$. It follows that
\begin{align*}
A_{12}=\gauss{2n-s-1}{n-s}q^{n-s}\cdot \gauss{n-1}{i-1}\cdot z_s\cdot z_{n-s}\cdot z_{i}\cdot z_{n-i}.
\end{align*}

\item We count the number $A_{21}$ of chambers $C$ with $C_{2n-s}\cap S=\{0\}$, $P\le C_n$ and $P\not\le C_{n-i}$.

There exist $q^{s(2n-1-s)}$ subspaces $U_{2n-s}$ of dimension $2n-s$ that are skew to $S$ and contain $P$. Each contains $\gauss{2n-s-1}{n-1}$ subspaces $U_n$ of dimension $n$ satisfying $P\le U_n$, and each such subspace $U_n$ contains $\gauss{n-1}{n-i}q^{n-i}$ subspaces $U_{n-i}$ of dimension $n-i$ with $P\not\le U_{n-i}$. Hence
\begin{align*}
A_{21}=q^{s(2n-1-s)}\cdot \gauss{2n-s-1}{n-1}\cdot \gauss{n-1}{n-i}q^{n-i}\cdot z_{n-i}\cdot z_i\cdot z_{n-s}\cdot z_{s}.
\end{align*}

\item We count the number $A_{22}$ of chambers $C$ with $C_{2n-s}\cap S=\{0\}$ and $P\le C_{n+i}$ and $P\not\le C_n$.

\textbf{Case 1.} $s+i\le n$.

Then $n+i\le 2n-s$ and hence each chamber $C$ with $C_{2n-s}\cap S=\{0\}$ also satisfies $C_{n+i}\cap S=\{0\}$. In the quotient space on $P$, we see that there are
\begin{align*}
\gauss{2n-s-1}{n+i-1}q^{s(n+i-1)}
\end{align*}
subspaces $U_{n+i}$ satisfying $P\le U_{n+i}$ and $U_{n+i}\cap S=\{0\}$. Each such subspace $U_{n+i}$ contains
\begin{align*}
\gauss{n+i-1}{n}q^{n}
\end{align*}
subspaces $U_n$ with $P\not\le U_n$, and each subspace $U_{n+i}$ is contained in
\begin{align*}
q^{s(2n-s-(n+i))}
\end{align*}
subspaces $U_{2n-s}$ with $U_{2n-s}\cap S=\{0\}$. Hence for $s\le n-i$ we have
\begin{align*}
A_{22}=\gauss{2n-s-1}{n+i-1}
\gauss{n+i-1}{n}    \cdot q^{s(2n-s-1)+n} \cdot
z_n\cdot z_{i}\cdot z_{2n-s-n-i}\cdot z_s.
\end{align*}

\textbf{Case 2.} $s+i\ge n$.

Then $n+i\ge 2n-s$. There exist $q^{s(2n-s)}$ subspaces $U_{2n-s}$ that intersect $S$ trivially, and of these $q^{s(2n-s-1)}$ contain $P$.

If $U_{2n-s}$ is a subspace skew to $S$ and if $P\in U_{2n-s}$, then $U_{2n-s}$ lies in $\gauss{2n-(2n-s)}{n+i-(2n-s)}$ subspaces $U_{n+i}$ of dimension $n+i$ and all these contain $P$, and $U_{2n-s}$ contains $\gauss{2n-s-1}{n}q^n$ subspaces of dimension $n$ that all do not contain $P$.

If however $U_{2n-s}$ is a subspace skew to $S$ and if $P\notin U_{2n-s}$, then $U_{2n-s}$ lies in $\gauss{2n-(2n-s+1)}{n+i-(2n-s+1)}$ subspaces $U_{n+i}$ of dimension $n+i$ that contain $P$, and $U_{2n-s}$ contains $\gauss{2n-s}{n}$ subspaces of dimension $n$ and none of these contains $P$.

Hence, the number of flags $(U_n,U_{2n-s},U_{n+i})$ that can be extended to a chamber satisfying the requirements in question is \begin{align*}
&(q^{s(2n-s)}-q^{s(2n-s-1)})
\gauss{2n-(2n-s+1)}{n+i-(2n-s+1)}\gauss{2n-s}{n}
\\
&+q^{s(2n-s-1)}\gauss{2n-(2n-s)}{n+i-(2n-s)}\gauss{2n-s-1}{n}q^n.
\end{align*}
Since each such flag can be extended in exactly $z_nz_{n-s}z_{n+i-(2n-s)}z_{n-i}$ ways to a chamber, which all fulfill the requirements, we find after some simplifications of the above term that the total number of chambers in question is equal to
\begin{align*}
A_{22}=
q^{s(2n-s-1)} &\left((q^s-1)
\gauss{s-1}{s+i-n-1}\gauss{2n-s}{n}+\gauss{s}{s+i-n}\gauss{2n-s-1}{n}q^n\right)z_nz_{n-s}z_{i+s-n}z_{n-i}.
\end{align*}
\end{enumerate}
We have
\begin{align}\label{eqn_antidesign_finalequation}
v_S\chi^i_P=q^{s(2n-s)-n+i}A_{11}-q^{s(2n-s)-n}A_{12}+q^iA_{21}-A_{22}.
\end{align}
To see that this is zero, we first consider the case that $i+s\le n$. Using Lemma \ref{firstremarks} we find
\begin{align*}
\frac{A_{11}q^i}{A_{12}}&=\frac
{\gauss{2n-s-1}{n-s-1}}
{\gauss{2n-s-1}{n-s}}
\cdot
\frac
{\gauss{n-s-1}{i-1}}
{\gauss{n-1}{i-1}}
\cdot
\frac
{z_{n-i-s}\cdot z_n}
{z_{n-s}\cdot z_{n-i}}
\\
&=\frac
{[2n-s-1]\cdot [n-s]\cdot [n-1]}
{[n-s-1]\cdot [n]\cdot [2n-s-1]}
\cdot
\frac
{[n-s-1]\cdot [i-1]\cdot [n-i]}
{[i-1]\cdot [n-s-i]\cdot [n-1]}
\cdot
\frac
{[n-i-s]\cdot [n]}
{[n-s]\cdot [n-i]}
=1
\end{align*}
and
\begin{align*}
\frac{A_{21}q^i}{A_{22}}&=\frac
{\gauss{2n-s-1}{n-1}\cdot \gauss{n-1}{n-i}\cdot z_{n-i}\cdot z_{n-s}}
{\gauss{2n-s-1}{n+i-1}
\gauss{n+i-1}{n}\cdot z_n\cdot z_{n-s-i}}
\\
&=
\frac
{[2n-s-1]\cdot [n+i-1]\cdot [n-s-i]\cdot [n-1]\cdot [n]\cdot [i-1]\cdot [n-i]\cdot [n-s]}
{[n-1]\cdot [n-s]\cdot [2n-s-1]\cdot [n-i]\cdot [i-1]\cdot [n+i-1]\cdot [n]\cdot [n-s-i]}=1.
\end{align*}
It follows that the right-hand side of \eqref{eqn_antidesign_finalequation} is zero. Now consider the case when $i+s\ge n$.
Then

\begin{align*}
& \hspace{-2ex} \frac{(q-1)^{n+s}}{z_{n-s}q^{s(2n-1-s)}}\cdot\left(q^{s(2n-s)-n+i}A_{11}+A_{21}q^i\right)
\\
&=\gauss{2n-s-1}{n-s-1}\cdot q^{i+s-n}z_nz_s
 +q^{n}\cdot \gauss{2n-s-1}{n-1}\cdot \gauss{n-1}{n-i}\cdot z_{n-i}\cdot z_i\cdot  z_{s}
\\
&=\frac{[2n-s-1]\cdot q^{i+s-n}[n][s]}
{[n-s-1]\cdot [n]}
 +\frac{q^{n}\cdot [2n-s-1][n-1]\cdot[n-i][i][s]}
 {[n-1][n-s][n-i][i-1]}
\\
&=\frac{[2n-s-1]\cdot q^{i+s-n}[s]}
{[n-s-1]}
 +\frac{q^{n}\cdot [2n-s-1][i][s]}
 {[n-s][i-1]}
\\
&=\frac{[2n-s-1][s]}
{[n-s]}\left(
q^{i+s-n}(q^{n-s}-1)
 +q^{n}(q^i-1)
\right)
\\
&=\frac{[2n-s-1][s]}
{[n-s]}\left(
q^i-q^{i+s-n}+q^{n+i}-q^n
\right)
\end{align*}
and
\begin{align*}
& \hspace{-2ex} \frac{(q-1)^{n+s}}{z_{n-s}q^{s(2n-1-s)}}q^{s(2n-s)-n}A_{12}+A_{22}
\\
&=\gauss{2n-s-1}{n-s}\cdot \gauss{n-1}{i-1}\cdot [s]\cdot [i]\cdot [n-i]
\\
&+\left((q^s-1)
\gauss{s-1}{s+i-n-1}\gauss{2n-s}{n}+\gauss{s}{s+i-n}\gauss{2n-s-1}{n}q^n\right)[n]\cdot [i+s-n][n-i]
\\
&=\frac{[2n-s-1]\cdot [s]\cdot [i]}
{[n-s][i-1]}
+(q^s-1)\frac{[s-1][2n-s]\cdot [i+s-n]}
{[s+i-n-1][n-s]}
+
\frac{[s][2n-s-1]q^n}
{[n-s-1]}
\\
&=\frac{[2n-s-1][s]}
{[n-s]}\left(
\frac{[i]}
{[i-1]}
+\frac{(q^{2n-s}-1)\cdot [i+s-n]}
{[s+i-n-1]}
+\frac{q^n(q^{n-s}-1)}
{1}
\right)
\\
&=\frac{[2n-s-1][s]}
{[n-s]}\left(
(q^i-1)
+(q^{2n-s}-1)(q^{s+i-n}-1)
+q^n(q^{n-s}-1)
\right)
\\
&=\frac{[2n-s-1][s]}
{[n-s]}\left(q^i+q^{n+i}-q^{s+i-n}-q^n\right)
\end{align*}
and hence also in this case the right-hand side of \eqref{eqn_antidesign_finalequation} is zero.
\end{proof}

\begin{theorem} \label{T: subspace antidesign}
For every nontrivial subspace $S$ of dimension at most $n$ of $\ff_q^{2n}$, the map $v_S$ defined in \ref{SubspaceAntidesign} is an antidesign. If $s=\dim(S)$, then $\allones^\top v_S=q^{s(2n-s)-n}(q^{n}+1)z_sz_{2n-s}$
\end{theorem}
\begin{proof}
Since the eigenspace for the smallest eigenvalue of $\Gamma_{2n}(q)$ is spanned by the vectors $\chi^i_P$, Lemma \ref{L: subspace antidesign 1} and Lemma \ref{L: subspace antidesign 2} show that $v_S$ is an antidesign.

The formula for $\allones^\top v_S$ follows from the definition of $v_S$ and the observation that the number of chambers $C$ satisfying $C_s=S$ is $z_sz_{2n-s}$ and that the number of chambers satisfying $S\cap C_{2n-s}=\{0\}$ is $q^{s(2n-s)}z_sz_{2n-s}$.
\end{proof}

\begin{cor} \label{L: preparation for EKR proof}
If $S$ is a subspace of dimension $s\le n$ and if $\EKR$ is a maximum EKR-set of chambers of $\ff_q^{2n}$, then $v_S^\top \charvec_{\EKR}=q^{s(2n-s)-n}z_sz_{2n-s}$.
\end{cor}
\begin{proof}
Since $\allones^\top\charvec_\EKR=\gauss{2n-1}{n}z_n^2$ and since the total number of chambers is $z_{2n}$, Theorem \ref{T: subspace antidesign} shows that
\begin{align*}
\charvec_\EKR^\top v_S&=\frac{\gauss{2n-1}{n}z_n^2\cdot q^{s(2n-s)-n}(q^{n}+1)}{z_{2n}}\cdot z_sz_{2n-s}
\\
&=\frac{\gauss{2n-1}{n}[n]^2\cdot q^{s(2n-s)-n}(q^{n}+1)}{[2n]}\cdot z_sz_{2n-s}
\\
&=\frac{[2n-1][n]\cdot q^{s(2n-s)-n}(q^{n}+1)}{[2n][n-1]}\cdot z_sz_{2n-s}
\\
&=\frac{[n]\cdot q^{s(2n-s)-n}(q^{n}+1)}{(q^{2n}-1)[n-1]}\cdot z_sz_{2n-s}
\\
&=q^{s(2n-s)-n}z_sz_{2n-s}.
\end{align*}
\end{proof}

 \section{Classifying maximum EKR-sets of chambers} \label{Section: Classification of ekr}

In this section we acquire all results needed to prove Theorem \ref{thm:Bound}.

\begin{notation} Let $\EKR$ be an EKR-set of chambers of $\ff_q^{2n}$.
\begin{enumerate}[1.]
\item For every subspace $S$ of $\ff_q^{2n}$ the weight of $S$ (with respect to $\EKR$) is the number of chambers of $\EKR$ that contain $S$. We call $S$ \emph{heavy} (with respect to $\EKR$) if all chambers that contain $S$ belong to $\EKR$; otherwise, we call $S$ \emph{light}.
\item A chamber is called \emph{light} if all its subspaces are light, and otherwise it is called \emph{heavy}.
\end{enumerate}
\end{notation}

\begin{lemma} \label{C: key for the proof}
Let $0<s<2n$, let $S$ be a subspace of dimension $s$, and $\EKR$ a maximum EKR-set of chambers of $\ff_q^{2n}$. Consider the chambers $C \in \EKR$, let $x$ be the number of these with $C_s=S$, let $y$ the number of these with $C_{2n-s}\cap S=\{0\}$, and let $z$ the number of these with $C_s\neq S$ and $C_{2n-s}\cap S\neq \{0\}$.
\begin{enumerate}[(a)]
\item We have $y=q^{s(2n-s)-n}z_sz_{2n-s}-xq^{s(2n-s)-n}$.
\item If $S$ is heavy, then $x=z_sz_{2n-s}$ and $z=\frac{z_{2n}}{1+q^{n}}-z_sz_{2n-2}$.
\item If $S$ is light, then $x\le z_sz_{2n-s}-q^{n^2-n+(n-s)^2}$ and $z\leq \frac{z_{2n}}{1+q^n}-q^{2n^2-2n}$.
\end{enumerate}
\end{lemma}
\begin{proof}
It suffices to prove this for $s\le n$, since the remaining cases are dual to these ones.
Therefore, we assume $s\le n$. We have $x+y+z=|\EKR|=\frac{z_{2n}}{1+q^{n}}$. From the definition of $v_S$ (see Notation \ref{SubspaceAntidesign}) we see that $\charvec_\EKR^\top v_S=xq^{s(2n-s)-n}+y$ so Corollary \ref{L: preparation for EKR proof} proves the formula for $y$. It follows that $x\le z_sz_{2n-s}$. If equality holds, then $y=0$ and $z=|\EKR|-x$ and we are in situation (b).

Now consider the situation when $x<z_sz_{2n-s}$. Then $y$ is positive, so there exists a chamber $C\in \EKR$ such that $C_{2n-s}\cap S=\{ 0\}$. Lemma \ref{basicnumbers} shows that the number of chambers that contain $S$ and are opposite to $C$ is $q^{{s\choose 2}+{2n-s\choose 2}}=q^{2n^2-2ns-n+s^2}$. Since the total numbers of chambers containing $S$ is $z_sz_{2n-s}$, it follows that $x\le z_sz_{2n-s}-q^{2n^2-2ns-n+s^2}$. Using the formula for $y$ and the equation $x+y+z=\frac{z_{2n}}{1+q^{n}}$, it follows that we are in situation (c).
\end{proof}

First, we consider the case in which every chamber in $\EKR$ is heavy.

\begin{lemma} \label{L: heavy subspaces pairwise meet.}
Let $\EKR$ be a maximum EKR-set of chambers of $\ff_q^{2n}$ and let $1\leq s\leq 2n-1$.
\begin{itemize}
	\item[(a)] A subspace $S$ of dimension $s$ is heavy if and only if
for every chamber $C$ in $\EKR$ we have $C_{2n-s}\cap S\neq \{ 0\}$.
	\item[(b)] If $s\le n$, the heavy subspaces of dimension $s$ mutually meet non-trivially.
	\item[(c)] For $s\le n$ the number of heavy $s$-subspaces is at most $\Gauss{2n-1}{s-1}$ and the number  of heavy $(2n-s)$-subspaces is at most $\Gauss{2n-1}{s-1}$.
  \item[(d)] A heavy $n$-subspace meets the $n$-subspace of every chamber of $\EKR$.
\end{itemize}
\end{lemma}

\begin{proof}
\begin{itemize}
\item[(a)] The statement follows directly from Lemma \ref{C: key for the proof}.
\item[(b)] Assume that two heavy subspaces $S_1$ and $S_2$ of dimension $s\leq n$ meet trivially. As $S_1$ and $S_2$ meet trivially, there exists a chamber $C$ with $C_s=S_1$ and $C_{2n-s}\cap S_2=\{0\}$. As $S_1$ is heavy, then $C\in\EKR$. But this contradicts (a) applied to $S_2$.
  \item[(c)] For $s\le n$, the heavy $s$-subspaces mutually intersect non-trivially, so the EKR-theorem for vector spaces (\cite{ekr_vectorspaces}) shows that there exists at most $\Gauss{2n-1}{s-1}$ heavy $s$-subspaces. The second statement follows by duality.
  \item[(d)] Consider two skew $n$-subspaces $S_1$ and $S_2$ and assume that there is a chamber in $\EKR$ that contains $S_1$ and there is a chamber $C$ in $\EKR$ that contains $S_2$. If $S_1$ is heavy, all chambers that have $S_1$ as their $n$-subspace are in $\EKR$. One of these chambers is opposite to $C$, contradiction.
 \end{itemize}
\end{proof}

\begin{lemma} \label{L: properties of O}
	\begin{itemize}
	\item[(a)] For $n\geq k$ we have $\deg(\Gauss{n}{k})=k(n-k)$.\\
	\item[(b)] For $s\in \mathbb{N}_0$ we have $\deg(z_s)=\frac{s(s-1)}{2}$.
%	\item[c)] If $f(q)$ and $g(q)$ are two polynomials and $deg(f(q))>deg(g(q))$, there is an $m(f,g)\in \mathbb{N}$, such that $f(q)>g(q)$ for all $q\geq m(f,g)$.
	\end{itemize}
\end{lemma}

\begin{proof}
Part (a) follows directly from the definition of the Gaussian coefficient. For part (b) we have
\begin{align*}
	\begin{split}
		z_s=\Gauss{s}{1}\Gauss{s-1}{1} \ldots \Gauss{2}{1},
	\end{split}
\end{align*}
with (a) and Gauss summation we obtain
\begin{align*}
	\deg(z_s)=(s-1)+(s-2)+\ldots +1=\frac{s(s-1)}{2}.
\end{align*}
%Finally, c) is trivial.
\end{proof}

\begin{remark} \label{R: O size of EKR-set}
By Theorem \ref{thm:maxsize} the maximum EKR-sets of chambers of $\ff_q^{2n}$
 contain $\frac{z_{2n}}{1+q^n}$ elements. Lemma \ref{L: properties of O} shows that $\deg(\frac{z_{2n}}{1+q^n})={2n^2-2n}$. \\
\end{remark}

We need the following Hilton-Milner type theorems.

\begin{result} \label{L: HM_new}
Let $M$ be a set of pairwise intersecting $n$-subspaces in $\ff_q^{2n}$. Assume that there is no subspace $P$ of dimension $1$ or $2n-1$, such that all subspaces in $M$ contain $P$ or are contained in $P$.
\begin{itemize}
\item[(a)] (Theorem 6 in \cite{Ihringer_Hilton_Milner}) Let $q\geq 4$ and $n\geq 4$. Then $|M|<3\Gauss{n}{1}\Gauss{2n-2}{n-2}$.
\item[(b)] (Theorem 6.1 in \cite{k=3_Hilton_Milner})
  Let $n=3$. Then $|M|\leq q^5+2q^4+3q^3+2q^2+q+1$.
\end{itemize}
\end{result}

Note that for $n=3$, we have $q^5+2q^4+3q^3+2q^2+q+1=\Gauss{n}{1}\Gauss{2n-2}{n-2}-(q^2+q)$.

\begin{lemma} \label{L: omega}
Let $n\geq 3$ and $q\geq 4$ and suppose that $\EKR$ is a
maximum EKR-set of chambers of $\ff_q^{2n}$. Furthermore, suppose that $\EKR$ is not of classical type.  Then there exist less than $3\Gauss{n}{1}\Gauss{2n-2}{n-2}$ heavy $n$-subspaces.
\end{lemma}
\begin{proof}
Let $M$ be the set of heavy $n$-subspaces in $\EKR$.

\textbf{Case 1.} There exists a $1$-dimensional subspace $P$ that is contained in all subspaces of $M$.\\
Since $\EKR$ ist not of classical type, there exists a chamber in $\EKR$ whose $n$-subspace $S$ does not contain $P$. By Lemma \ref{L: heavy subspaces pairwise meet.} every subspace of $M$ meets $S$ non-trivially. The number of $2$-dimensional subspaces that contain $P$ and meet $S$ is $\Gauss{n}{1}$. The number of $n$-subspaces incident with a given $2$-subspace is $\Gauss{2n-2}{n-2}$. Together we get
\begin{align*}
 	|M|\leq \Gauss{n}{1} \Gauss{2n-2}{n-2}.
\end{align*}

\textbf{Case 2.} There exists a $(2n-1)$-dimensional subspace $P$ that contains all subspaces of $M$.\\
This is dual to Case 1.\\

\textbf{Case 3.} There does not exist a subspace $P$ as in Case 1 or Case 2.\\
Then Result \ref{L: HM_new} implies
 \begin{align*}
 	|M|\leq 3\Gauss{n}{1}\Gauss{2n-2}{n-2}.
  \end{align*}
\end{proof}

\begin{prop} \label{P: heavy chambers}
Let $n\geq 3$ and $q\geq 4$. Suppose that $\EKR$ is a maximum EKR-set. Furthermore, suppose that $\EKR$ is not classical
and that every chamber of $\EKR$ is heavy. Then there exists a constant $m$ depending on $n$ but not on $q$ such that  $|\EKR|<mq^{2n^2-2n-1}$.
\end{prop}
\begin{proof}
For $1\le s\le 2n-1$ let $\alpha_s$ be the number of heavy $s$-subspaces in $\ff_q^{2n}$. Since every chamber in $\EKR$ contains at least one heavy subspace, Lemma \ref{ChambeerExtensions} implies that
\begin{align*}
 |\EKR|  \leq \sum\limits_{i=1}^{2n-1} \alpha_s \cdot z_sz_{2n-s}=\alpha_nz_n^2+\sum_{s=1}^{n-1}(\alpha_s+\alpha_{2n-s})z_sz_{2n-s}.
\end{align*}
For $\alpha_n$ we use the upper bound from Lemma \ref{L: omega} and for the other $\alpha_s$ we use the upper bound from part (c) of Lemma \ref{L: heavy subspaces pairwise meet.} and find
\begin{align}\label{A: heavy chamber}
	|\EKR|< 3\Gauss{n}{1}\Gauss{2n-2}{n-2}z_n^2 + 2 \sum\limits_{s=1}^{n-1}\Gauss{2n-1}{s-1}  z_sz_{2n-s}.
\end{align}
For $1\le s<n$ Lemma \ref{L: properties of O} implies
\begin{align*}
	\deg(\Gauss{2n-1}{s-1}  z_sz_{2n-s})=2n^2-3n+s\le 2n^2-2n-1
\end{align*}
and
\begin{align*}
\deg(3\Gauss{n}{1}\Gauss{2n-2}{n-2}z_n^2)=q^{2n^2-2n-1}.
\end{align*}
Hence the right-hand side of \eqref{A: heavy chamber} is a polynomial in $q$ of degree $2n^2-2n-1$. The result follows.
\end{proof}

Now we consider the case in which $\EKR$ contains a light chamber.

\begin{prop} \label{P: light chamber}
	Let $\EKR$ be a maximum Erd\H os-Ko-Rado set of chambers of $\ff_q^{2n}$. Assume that $\EKR$ contains a light chamber. Then  there exists a constant $m$ such that  $|\EKR|<mq^{2n^2-2n-1}$.
\end{prop}

\begin{proof}
	Let $C$ be a light chamber in $\EKR$.
	Every chamber $B=(B_1,\ldots,B_{2n-1})$ in $\EKR$ has to satisfy $B_{2n-s}\cap C_s\neq \{0 \}$ for some $1\leq s\leq 2n-1$. Let $x^s$ be the number of chambers $B$ with $B_s=C_s$ and let $z^s$ be the number of chambers $B$ with $B_s\neq C_s$ and $B_{2n-s}\cap C_s\neq \{0 \}$. Then
	\begin{align} \label{A: |F| light chamber}
		|\EKR|\leq \sum\limits_{s=1}^{2n-1} (x^s+z^s).
	\end{align}
Lemma \ref{C: key for the proof} implies
	\begin{align*}
	\begin{split}
		x^s&\leq z_sz_{2n-s}-q^{2n^2-2ns-n+s^2}, \\
		z^s&\leq \frac{z_{2n}}{1+q^n}-q^{2n^2-2n}.
	\end{split}
	\end{align*}
	 First, we consider $x^s$. From Lemma \ref{L: properties of O} we have
	 \begin{align*}
	 	\deg(z_sz_{2n-s})& = \frac{s(s-1)}{2}+\frac{(2n-s)(2n-s-1)}{2}=2n^2-2ns-n+s^2.
	 \end{align*}
Recall $1\leq s\leq 2n-1$. Since the leading coefficient of $z_sz_{2n-s}$ is $1$, we get
\begin{align*}
	\deg(z_sz_{2n-s}-q^{2n^2-2ns-n+s^2})\leq 2n^2-2ns-n+s^2-1\leq 2n^2-3n+1.
\end{align*}
Now we consider $z^s$. We have
	 \begin{align*}
	 \deg(\frac{z_{2n}}{1+q^n})&=n(2n-1)-n=2n^2-2n.
		\end{align*}
	Since the leading coefficient of $\frac{z_{2n}}{1+q^n}$ is $1$, we get
\begin{align*}
	\deg(\frac{z_{2n}}{1+q^n}-q^{2n^2-2n})\leq 2n^2-2n-1.
\end{align*}
Using (\ref{A: |F| light chamber}) yields the desired result.
\end{proof}

Theorem \ref{thm:Bound} follows directly from Remark \ref{R: O size of EKR-set}, Proposition \ref{P: heavy chambers} and \ref{P: light chamber}.

\begin{remark} 
For small values of $n$ a precise investigation of the polynomial inequalities (\ref{A: heavy chamber}) and (\ref{A: |F| light chamber}) enables us to give an explicit value for $m_n$ occurring in Theorem \ref{thm:Bound}, as given in the following table. \begin{table}[htbp]
\begin{tabular}{lllllll}
\toprule
$n$ & 3 & 4 & 5 & 6 & 7& 8 \\ 
$m_n$ & 23& 46& 77& 116 &163& 218\\ 
\bottomrule
\end{tabular}
\end{table}
\end{remark}

\section{The case $n=2$}

We have shown for $n\geq 3$ that every maximum EKR-set of chambers is of classical type for  sufficiently large $q$. In this section we prove the same result for $n=2$ and all $q$. In order to show this, we need to closely examine the geometric properties of chambers of $\ff_q^4$. We use observations made in \cite{heeringmetsch2023secondmax}. Since \cite{heeringmetsch2023secondmax} utilizes the language of projective spaces, we switch to projective notation for this section. A chamber of $\ff_q^4$ is now a chamber of $\PG(3,q)$ that consist of a (projective) point, a (projective) line and a (projective) plane that are mutually incident. 
We adopt some notation from \cite{heeringmetsch2023secondmax}.

\begin{notation}
\begin{enumerate}[1.]
\item Chambers of $\PG(3,q)$ consisting of a point $P$, a line $\ell$ and a plane $\pi$ are written as triples $(P,\ell,\pi)$.
\item The chambers of an EKR-set $\EKR$ are also called \emph{$\EKR$-chambers}.
\item  If $\ell$ is a line, then the number of chambers of an EKR-set $\EKR$ that contain $\ell$ is called the \emph{weight} of $\ell$.
\end{enumerate}
\end{notation}

\begin{lemma}  \label{P: weight of lines}
Let $\EKR$ be a maximum independent set of chambers of $\PG(3,q)$ and let $\ell$ be a line.
Then $\ell$ has weight $0$, $1$, $2$, $q+1$, $2q+1$ or $(q+1)^2$. Moreover, the following hold.
\begin{enumerate}[\rm (a)]
\item The line $\ell$ has weight $(q+1)^2$ if and only if it meets the line of every $\EKR$-chamber.
\item If $\ell$ has weight $2q+1$, then $\ell$ is incident with a point $P$ and a plane $\pi$ such that the $\EKR$-chambers that contain $\ell$ are the chambers that contain $\ell$ and $P$ or $\pi$. Furthermore every $\EKR$-chamber $(Q,h,\tau)$ with $h\cap\ell=\emptyset$ satisfies $Q\in\pi$ and $P\in \tau$. In particular, every line $h$ of a chamber in $\EKR$ with $h\cap\ell=\emptyset$ has weight $1$.
\item If $\ell$ has weight $q+1$, then one of the following two cases occurs.
\begin{enumerate}[(i)]
\item There exists a plane $\pi$ on $\ell$, such that the $\EKR$-chambers of $\ell$ are the $q+1$ chambers that contain $\pi$ and $\ell$. In this case every $\EKR$-chamber $(Q,h,\tau)$ satisfies $h\cap\ell\not=\emptyset$ or $Q\in\pi$.
\item There exists a point $P$ on $\ell$, such that the $\EKR$-chambers of $\ell$ are the $q+1$ chambers that contain $P$ and $\ell$. In this case every $\EKR$-chamber $(Q,h,\tau)$ satisfies $h\cap\ell\not=\emptyset$ or $P\in\tau$. 
    \end{enumerate}
\end{enumerate}
\end{lemma}

\begin{proof}
This is Proposition 2.3 and Notation 2.8 in \cite{heeringmetsch2023secondmax}.
 \end{proof}

 \begin{notation}
Let $\ell$ be a line of weight $<(q+1)^2$. If there is a plane $\tau$ such that all the $q+1$ chambers on $\ell$ and $\tau$ are in $\EKR$, we call $\ell$ a $\pi$-line.
If there is a point $Q$ such that all the $q+1$ chambers on $\ell$ and $Q$ are in $\EKR$, we call $\ell$ a $P$-line.
\end{notation}

 \begin{lemma} \label{L: requirements}
 \begin{enumerate}[\rm (a)]
 	\item A line of weight $2q+1$ is a $\pi$-line and a $P$-line.
 	\item No two $\pi$-lines are skew.
 	\item No two $P$-lines are skew.
 	\item There are at most two pairwise skew lines of weight $2$.
 	\item A line of weight $2$ meets every line of weight $>2$.
 \end{enumerate}
 \end{lemma}

\begin{proof}
Part (a) follows directly from Lemma \ref{P: weight of lines}. Part (b) and (c) are dual, and (b) follows from Lemma 2.9 in \cite{heeringmetsch2023secondmax}. 
Lemma 5.3 in \cite{heeringmetsch2023secondmax} implies (c)
and Lemma 5.6 in \cite{heeringmetsch2023secondmax} is part (d).
\end{proof}

We restate Theorem \ref{T: chambers of f_q^4} in projective language.

\begin{thm} \label{P: weight in spreads}
  Let $\EKR$ be a maximum EKR-set of chambers of $\PG(3,q)$. Then for any two chambers $(P_1,\ell_1,\pi_1)$ and $(P_2,\ell_2,\pi_2)$ in $\EKR$, we have that $\ell_1$ and $\ell_2$ are not skew.
\end{thm}

\begin{proof}
Assume that two lines of chambers of $\EKR$ are skew. Extend these two skew lines to a line-spread $S=\{\ell_0,\ldots,\ell_{q^2}\}$ of $\PG(3,q)$ and
let $L$ be the set of all chambers of $\PG(3,q)$ whose line is in $S$.
Let $v_S$ be a map from the set of chambers of $\PG(3,q)$ to $\mathbb{Q}$, a chamber gets mapped to $1$ if it is in $L$ and to $0$ if it is not in $L$.
Theorem \ref{thm:SpreadAntidesigns} yields that $v_S$ is an antidesign.
Corollary \ref{C: spreads intersection size} yields $|\EKR \cap L|=(q+1)^2$.\\
Any line $\ell_i\in S$ is skew to at least one line of a chamber in $\EKR$.
Therefore Lemma \ref{P: weight of lines} implies that $\ell_i$ has weight $0$, $1$, $2$, $q+1$, or $2q+1$. We denote the weight of $\ell_i$ as $w(\ell_i)$. 
We have
\begin{align} \label{A: Proof chambers 3d}
\sum\limits_{i=0}^{q^2} w(\ell_i)=|\EKR \cap L| =(q+1)^2.
\end{align}
Assume that there is no $\pi$-line in $S$. Then we have two options. There is a $P$-line in $S$, or there is no $P$-line in $S$. \\
If there is a $P$-line in $S$, Lemma \ref{L: requirements} implies that there is at most one $P$-line
in $S$ and no line of weight $2$ in $S$. Note that in this case the $P$-line has to have weight $q+1$, since there is no $\pi$-line in $S$. This implies 
$$
\sum\limits_{i=0}^{q^2} w(\ell_i)\leq (q+1)+\sum\limits_{i=0}^{q^2-1} 1<(q+1)^2
$$
which is a contradiction to (\ref{A: Proof chambers 3d}).\\
If there is no $P$-line in $S$, Lemma \ref{L: requirements} implies that all lines in $S$ have weight $1$, except at most two lines, which can have weight $2$. This implies 
$$
\sum\limits_{i=0}^{q^2} w(\ell_i)\leq 2+2+\sum\limits_{i=0}^{q^2-2} 1<(q+1)^2
$$
which also stands in contradiction to (\ref{A: Proof chambers 3d}).\\
We have shown that there is a $\pi$-line in $S$. Using duality and Lemma \ref{L: requirements} we get  that there is exactly one $\pi$-line in $S$ and exactly one $P$-line in $S$, note that the $\pi$-line and the $P$-line might be the same line. It follows from (\ref{A: Proof chambers 3d}) that all the lines in $S$ that are neither a $P$-line, nor a $\pi$-line have to have weight $1$.\\
From Lemma \ref{P: weight of lines} we know that for every $\pi$-line, there is a special plane, such that all chambers of $\EKR$, whose line is skew to the $\pi$-line, have their point in this special plane. For the $P$-line we get the dual statement. Let $\pi$ be the special plane associated with the $\pi$-line in $S$ and let $P$ be the special point associated with the $P$-line in $S$. Lemma \ref{P: weight of lines} implies $P\in\pi$.\\
There are at least $q^2-1$ lines $\ell_0,\ldots,\ell_{q^2-2}$ in $S$ that have weight one and that are skew to the $\pi$-line and the $P$-line. Let $(Q_i,\ell_i,\tau_i)$ be the chambers of $\EKR$ associated with $\ell_i$ for $i=0,\ldots,q^2-2$. Since the lines of $S$ are pairwise skew, Lemma \ref{P: weight of lines} implies that $Q_i\in\pi$ and $P\in \tau_i$ for all $i=0,\ldots,q^2-2$.
Let $h$ be a line incident with $P$ and $\pi$. There are $q$ points $\neq P$ on $h$. Since $q^2-1>q$ for all $q\geq 2$, there are two points (w.l.o.g.) $Q_0$ and $Q_1$ in $\pi$, such that $P$, $Q_0$ and $Q_1$ are not collinear. But since $P\in\tau_0$, we have $\tau_0\cap \pi=PQ_0$ and hence $Q_1\notin \tau_0$. Similarly, we get $Q_0\notin \tau_1$. Since $\ell_0$ and $\ell_1$ are skew, we get that the chambers $(Q_0,\ell_0,\tau_0)$ and $(Q_1,\ell_1,\tau_1)$ are opposite, which is a contradiction.
\end{proof}

\section*{Acknowledgement}
The second author was supported by the German Academic Exchange Service (DAAD) [Research Grants - Short-Term Grants, 2023].

\bibliographystyle{plainurl}
\bibliography{EKRs-and-Antidesigns_arXiv.bib}

\end{document}